\documentclass[11pt]{article}
\usepackage{graphics,epsfig}
\usepackage{setspace}

\usepackage{amssymb,amsmath,bbm}
\usepackage{amsfonts,amsthm}
\usepackage[ansinew]{inputenc}
\usepackage{enumerate}
\usepackage{multirow}
\newtheorem{theorem}{Theorem}[section]
\newtheorem{corollary}[theorem]{Corollary}
\newtheorem{lemma}[theorem]{Lemma}

\newenvironment{remark}{\addtocounter{theorem}{1}\vskip 0.2cm{\sc Remark
\thetheorem.}}{\hfill\vskip 0.2cm}

\newcommand{\argmin}{\operatornamewithlimits{argmin}}

 \newcommand{\sgn}{\operatorname{sgn}}

\numberwithin{equation}{section}

\title{A Class of Solvable Stationary Singular Stochastic Control Problems}
\author{Luis H. R. Alvarez E.\thanks{Department of Accounting \& Finance, Turku School of Economics, FIN-20014 University of Turku, Finland, e-mail:
luis.alvarez@tse.fi } }
\begin{document}
\maketitle

\abstract{We consider the determination of the optimal stationary singular stochastic control of a linear diffusion for a class of average cumulative cost minimization problems arising in various financial and economic applications of stochastic control theory. We present a set of conditions under which the optimal policy is of the standard local time reflection type and state the first order conditions from which the boundaries can be determined. Since the conditions do not require symmetry or convexity of the costs, our results cover also the cases where costs are asymmetric and non-convex. We also investigate the comparative static properties of the optimal policy and delineate circumstances under which higher volatility expands the continuation region where utilizing the control is suboptimal.}\\

\noindent {\bf Keywords:} stationary singular stochastic control, optimal reflection, linear diffusions.\\

\noindent
{\bf JEL Subject Classification:} G35, G31, C44, Q23.\\

\noindent
{\bf AMS Subject Classification:} 93E20, 60G10, 49J10, 49K10.\\

\thispagestyle{empty} \clearpage \setcounter{page}{1}

\section{Introduction}

Many economic and financial decision making problems can be formulated as singular stochastic control problems. For example, the determination of the harvesting policy maximizing the expected cumulative present value of the harvesting yield in the case where the harvesting effort is unbounded constitutes a singular stochastic control problem (cf. \cite{AlSh,He_et_al,LES1,LES2,LO,SoStZh}). Analogously, cash flow management studies investigating optimal dividend distribution, recapitalization, or both in the presence of risk neutrality result into singular control problems (cf. \cite{AlVi,AsTa1,CaSaZa,ChTaZh,HoTa1,HoTa2,JS,MR,Paul,PeKe,SeTa,ShLeGa,Ta,TaZh}). Based on its applicability singular stochastic control has attracted much interest in mathematics and there exists a vast literature focusing on it (cf. \cite{A5,Ba,Bank,BaCh,BaEg,BSW,BK,DuMi1,DuMi2,Fer,GuoTo1,GuoTo2,HaTa,
HaSu1,HaSu2,HeSt,Jacka,Kar1,Kar2,KS1,KS2,KuSt1,KuSt,Ma2012,MR1,MR2,O}). However, ergodic singular stochastic control problems has not been investigated to the same extent. This is particularly surprising from the natural resource management point of view since ergodic controls play a central role in the analysis of  sustainable harvesting policies (see, for example, Chapter 1 in the seminal textbook \cite{Cl}). Since singular harvesting policies are associated with the situation where the admissible harvesting effort is unbounded, analyzing the optimal ergodic singular harvesting policy provides simultaneously insights on the long run sustainability of potentially bounded policies as well.

Motivated by the argument above, the objective of this study is to analyze and solve a class of ergodic singular stochastic control problems of linear diffusions. To this end, we reconsider under different assumptions the singular control problem originally introduced in \cite{JaZe} and analyze a class of average cumulative cost minimization problems where the controlled process is a regular linear diffusion and the costs associated with controlling the process are not necessarily symmetric nor convex. In this way, we plan to investigate how the potential asymmetry or nonconvexity of operative costs affect the optimal policy. Instead of analyzing the entire class of admissible policies at once we follow the analysis developed in \cite{Ha, ShLeGa} and focus on local time control policies maintaining the controlled process between two arbitrary boundaries and derive an explicit representation of the long run expected average costs. Following then the analysis of the pioneering studies \cite{Ha,Kar1,Ta1} considering the ergodic singular control of Brownian motion and \cite{MR3} considering the ergodic singular control of linear diffusions in the presence of smooth costs, we also characterize the stationary value in terms of an associated free boundary problem. We solve the free boundary problem and delineate a set conditions on the cost under which the optimal policy is to reflect the underlying diffusion at two constant boundaries constituting the unique solutions of the optimality conditions.
Since the conditions impose only relatively weak monotonicity and limiting conditions on the costs, our findings indicate that the optimality of ordinary reflection policies does not require the symmetry nor the convexity of costs; a result which is in line with the findings of \cite{JaZe}. We illustrate our general observations in explicitly parameterized examples and characterize circumstances under which our assumptions on the limiting behavior of the cost and drift coefficient can be weakened even further. Interestingly, we notice that the
assumed continuity of the drift coefficient is not always needed for our findings to remain valid and illustrate this explicitly in the case where the controlled process is a Brownian motion with alternating drift. We also investigate the impact of increased volatility on the optimal policy and its value and find that the convexity of the value of the optimal policy implies that higher volatility postpones the optimal exercise of the control policy by expanding the continuation region. In order to extend previous findings based on ordinary Brownian motion and to simultaneously illustrate explicitly the positive dependence between volatility and the incentives to postpone the utilization of the control policy, we present a class of models resulting into symmetric reflection policies and characterize a special case where the optimal boundaries are linearly dependent on the volatility of the underlying. Finally, in order to cover most standard models arising in the literature applying singular control of linear diffusions, we also study the cases where the underlying can be controlled only to one direction and establish conditions under which the optimal policy can be characterized as an instantaneous reflection policy at a single optimal boundary.

It is at this point worth pointing out that the ergodic singular stochastic control of a linear diffusion has been previously studied in \cite{We1,JaZe,We2,HeStZh} (see also \cite{JaZe1} for the associated impulse control model). \cite{We1} investigates the optimal stationary singular policy minimizing the long-term average cost in a setting where the controlling costs are symmetric and the decision maker chooses optimally also the drift coefficient of the underlying. \cite{JaZe}, in turn, investigates the optimal ergodic singular policy minimizing the long-term average cost in the presence of state-dependent costs. They delineate a set of general conditions under which the problem admits a unique solution and present a set of equations characterizing the optimal boundaries at which the underlying diffusion should be optimally reflected. One of the key findings of \cite{JaZe} is that the convexity of costs is not necessary for the existence of an optimal policy. \cite{We2} investigates the relationship between the value of the optimal singular policy minimizing the expected cumulative present value of the costs and the value of the optimal ergodic policy. He extends the Abelian limit result originally established in \cite{Kar1} for Brownian motion and states a set of general conditions under which the same conclusion is valid for a large class of linear diffusions as well.
Finally, \cite{HeStZh} investigates the optimal impulse control policy minimizing the long-term average cost in the case where the underlying is modeled as a one-dimensional diffusion. They establish conditions for the optimality of a standard $(s, S)$-ordering policy for a large class of cost functions. The analysis in \cite{HeStZh} is related to the one sided ergodic singular control problem considered in this study, since they also allow for potentially asymmetric holding/back-order and ordering costs within an ergodic impulse control setting. As is known from the literature on impulse control in the case where the objective is to optimize the expected cumulative present value of future revenues, the optimal exercise boundaries of impulse control policies tend towards the ones associated to singular control as the transaction/ordering costs tend to zero (cf. \cite{MR2, O}). We notice that the same conclusion is valid in the single boundary setting as well and find that the optimality conditions characterizing the optimal boundaries developed in \cite{HeStZh} coincide with ours in the limit.

Ergodic singular stochastic control and ergodic control of singular processes has also been previously studied by relying on alternative and more general approaches than the one utilized in this study. \cite{KuSt1} analyzes a large class of stationary control problems by extending the findings of the studies \cite{St1,St2} considering the long run average cost minimization for a broad class of processes and controls satisfying a martingale problem defined in terms of the generator of the processes.
\cite{HeSt}, in turn, investigates based on the findings of \cite{St2} the numerical implementation of an infinite-dimensional linear programming
formulation of stochastic control problems involving singular
stochastic processes.  \cite{KuSt} develops conditions under which the optimal control of processes having both absolutely continuous and singular controls can be equivalently formulated as linear programs over a space of measures on the state and control spaces.

The contents of this study are as follows. The general two boundary problem is presented and solved in section 2. The results of the general analysis are then illustrated in explicitly parameterized models in Section 3. The single boundary cases are treated and illustrated in Section 4. Finally, Section  5 concludes our study.

\section{The Stationary Singular Stochastic Control Problem}
Let $\{X^{Z}_t; t\geq 0\}$ be a process defined on a complete
filtered probability space $(\Omega, \mathbb{P}, \{{\cal
F}_{t}\}_{t\geq 0}, {\cal F})$ and evolving on $\mathbb{R}$ according to the dynamics characterized by the stochastic differential
equation
\begin{align}
dX^{Z}_t = \mu(X^{Z}_t) dt + \sigma(X^{Z}_t)dW_t + dU_t -
dD_t, {\hskip .1in}X^{Z}_0 = x\in \mathbb{R},\label{sde}
\end{align}
where the drift coefficient $\mu:\mathbb{R}\mapsto\mathbb{R}$ is assumed to be continuous on $\mathbb{R}$. We also assume that the diffusion coefficient
$\sigma:\mathbb{R}\mapsto\mathbb{R}_+$ is continuous and satisfies the nondegeneracy condition $\sigma^2(x)>0$ for all $x\in \mathbb{R}$ and that the mapping
$x\mapsto (1+|\mu(x)|)/\sigma^{2}(x)$ is locally integrable on $\mathbb{R}$. As is known from the literature on stochastic differential equations, these conditions
guarantee the existence of a weak solution defined up to an explosion time for
the stochastic differential equation \eqref{sde} in the absence of
interventions (cf. \cite{KaSh}, Section 5.5). In \eqref{sde} $Z_t=(U_t,D_t)$ denote the applied control policies.
We call a control policy
admissible if it is non-negative, non-decreasing,
right-continuous, and $\{{\cal F}_{t}\}$-adapted, and denote the
set of admissible controls as $\Lambda$.

In accordance with the classical theory on linear diffusions, we define the differential operator ${\cal A}$ representing the infinitesimal generator of the underlying
diffusion $X$ as
\begin{align}
{\cal A} = \frac{1}{2}\sigma^{2}(x)\frac{d^2}{dx^2} +
\mu(x)\frac{d}{dx}\label{do}
\end{align}
and denote as
$$
S'(x) = \exp\left(-\int^x \frac{2\mu(y)dy}{\sigma^2(y)}\right)
$$
the density of the scale of $X$ and as
$$
m'(x) = \frac{2}{\sigma^2(x)S'(x)}
$$
the density of the speed measure of $X$.

Given the assumptions above, we now plan to consider the determination of the admissible pair of singular control policies $U_t,D_t\in \Lambda$ {\em minimizing} the expected long run average costs
\begin{align}
\limsup_{T\rightarrow \infty}\frac{1}{T}\mathbb{E}_x\int_0^T\left(c(X_s^Z)ds + q_u dU_s + q_d dD_s\right)\label{sscp}
\end{align}
where
$c:\mathbb{R}\mapsto\mathbb{R}_+$ is a continuous mapping measuring the flow of costs incurred from
operating with the prevailing stock, and $q_u,q_d \in \mathbb{R}_+$ are known parameters which can be interpreted as the unit cost of utilizing the associated control.  As usually, we assume that the costs $c$ are minimized at $0$ and, consequently, that $c(x)\geq c(0)\geq 0$ for all $x\in \mathbb{R}$.

In what follows the mappings $\pi_1(x):=c(x)+q_d \mu(x)$ and $\pi_2(x):=c(x)-q_u\mu(x)$ will play a central role in the determination of the optimal
stationary policy. We will assume throughout this section that these functions satisfy the following two conditions:
\begin{itemize}
  \item[(i)] there is a unique $\hat{x}_i=\argmin\{\pi_i(x)\}\in \mathbb{R}$ so that $\pi_i(x)$ is decreasing on $(-\infty, \hat{x}_i)$ and increasing on $(\hat{x}_i,\infty)$, where $i=1,2$, and
  \item[(ii)]  $\lim_{x\uparrow\infty}\pi_1(x) = \lim_{x\downarrow-\infty}\pi_2(x)=\infty$.
\end{itemize}
It is worth emphasizing that in the absence of drift, that is, when $\mu\equiv 0$, $\pi_1=\pi_2=c$ and $\hat{x}_1=\hat{x}_2=0$. In that case, assumption (i) essentially states that the operating costs $c$ needs to be decreasing below $0$ and increasing above it.

Having presented the considered ergodic control problem, we now follow the approach introduced in \cite{Ha} and \cite{ShLeGa}
(see also \cite{Kar1} and \cite{Ma2012}) and focus on the admissible reflecting barrier policies $U_t^a, D_t^b\in \Lambda$ which maintain
the process between two constant thresholds $a$ and $b$. Both policies are continuous on $t>0$ and increase only when
$X_t^{Z} = a$ and $X_t^{Z} = b$, respectively.  We can now establish the following auxiliary finding characterizing the
expected long-run average cumulative costs as a function of the boundaries $a$ and $b$.
\begin{lemma}\label{stationary}
Assume that $x\in [a,b]$. Then, the expected long-run average cumulative costs read as
\begin{align*}
\lim_{T\rightarrow\infty}\frac{1}{T}\mathbb{E}_x\int_0^T\left(c(X_s^Z)ds + q_u dU^{a}_{s} + q_d dD^{b}_{s}\right)=
C(a,b),
\end{align*}
where
\begin{align}
C(a,b) = \frac{1}{m(a,b)}\left[\int_a^bc(t)m'(t)dt+\frac{q_u}{S'(a)}+\frac{q_d}{S'(b)}\right].\label{aggregate}
\end{align}
Consequently, if a pair $(a^\ast,b^\ast)$ of boundaries minimizing the expected long-run average cumulative costs exists, it has to satisfy the ordinary first
order conditions
\begin{align}
C(a^\ast,b^\ast) &=\pi_2(a^\ast) \label{optimalitycond1a}\\
C(a^\ast,b^\ast) &=\pi_1(b^\ast).\label{optimalitycond2a}
\end{align}
which can be re-expressed as
\begin{align}
I_1(a^\ast,b^\ast):=\int_{a^\ast}^{b^\ast} (\pi_1(t)-\pi_1(b^\ast))m'(t)dt+\frac{q_d+q_u}{S'(a^\ast)} &=0 \label{optimalitycond1}\\
I_2(a^\ast,b^\ast):=\int_{a^\ast}^{b^\ast} (\pi_2(t)-\pi_2(a^\ast))m'(t)dt+\frac{q_d+q_u}{S'(b^\ast)} &=0.\label{optimalitycond2}
\end{align}
\end{lemma}
\begin{proof}
Let $f\in C^2(\mathbb{R})$ be an arbitrary twice continuously differentiable function and let $x\in [a,b]$. Applying the Doléans-Dade-Meyer change of variable formula to the function $f$ then yields
\begin{align*}
f(X_T^Z) = f(x) + \int_0^T(\mathcal{A}f)(X_s^Z)ds + \int_{0}^{T}
\sigma(X_s^Z)f'(X^{Z}_{s})dW_s + f'(a)U^{a}_{T} - f'(b)D^{b}_{T}
\end{align*}
since $X_t^Z$ is continuous on $[a,b]$ and the processes $U^{a}_{t}$ and $D^{b}_{t}$ increase only at the boundaries $a$ and $b$, respectively. Reordering terms, dividing with $T$, and taking expectations yields
\begin{align*}
\frac{\mathbb{E}_x[f(X_T^Z)]-f(x)}{T}= \frac{1}{T}\mathbb{E}_x\int_0^T(\mathcal{A}f)(X_s^Z)ds + f'(a)\frac{\mathbb{E}_x[U^{a}_{T}]}{T} - f'(b)\frac{\mathbb{E}_x[D^{b}_{T}]}{T},
\end{align*}
since $\sigma(x)f'(x)$ is bounded on $[a,b]$. Since $m(a,b) < \infty$ we find by letting $T\rightarrow \infty$ and invoking standard ergodic results for  linear diffusions that (cf. \cite{BS}, pp. 36--38; see also \cite{Ha}, pp. 89--92)
\begin{align*}
\int_a^b(\mathcal{A}f)(t)\frac{m'(t)}{m(a,b)}dt + f'(a)\alpha - f'(b)\beta = 0,
\end{align*}
where $\alpha = \lim_{t\rightarrow\infty}\mathbb{E}_x[U^{a}_{T}]/T$ and $\beta = \lim_{t\rightarrow\infty}\mathbb{E}_x[D^{b}_{T}]/T$. Choosing $f(x)=x$ and then $f(x)=S(x)$ yields the system of equations
\begin{align*}
m(a,b)(\beta-\alpha) &=\frac{1}{S'(b)}-\frac{1}{S'(a)}\\
S'(b)\beta -S'(a)\alpha &= 0
\end{align*}
from which we deduce that
$$
\beta=\frac{1}{S'(b)m(a,b)}\quad\quad\textrm{ and }\quad \quad \alpha = \frac{1}{S'(a)m(a,b)}
$$
which proves the alleged representation (see II.6.35 on p. 36 in \cite{BS}).

In order to establish the optimality conditions, we first find that standard differentiation of $C(a,b)$  yields
\begin{align*}
\frac{\partial C}{\partial a}(a,b)&=\frac{m'(a)}{m^2(a,b)}\left[\int_a^b c(t)m'(t)dt - \pi_2(a)m(a,b) + \frac{q_d}{S'(b)}+\frac{q_u}{S'(a)}\right]=\frac{m'(a)}{m(a,b)}\left[C(a,b)-\pi_2(a)\right]\\
\frac{\partial C}{\partial b}(a,b)&=\frac{m'(b)}{m^2(a,b)}\left[\pi_1(b)m(a,b) - \int_a^b c(t)m'(t)dt - \frac{q_d}{S'(b)}-\frac{q_u}{S'(a)}\right]=\frac{m'(b)}{m(a,b)}\left[\pi_1(b)-C(a,b)\right]
\end{align*}
proving the conditions \eqref{optimalitycond1a} and  \eqref{optimalitycond2a}.
Utilizing now the identity
\begin{align}
\int_a^b\mu(t)m'(t)dt=\frac{1}{S'(b)}- \frac{1}{S'(a)}\label{canonical}
\end{align}
demonstrates that the partial derivatives can be rewritten as
\begin{align}
\frac{\partial C}{\partial a}(a,b)&=\frac{m'(a)}{m^2(a,b)}\left[\int_a^b (\pi_2(t)-\pi_2(a))m'(t)dt +\frac{q_d+q_u}{S'(b)}\right]\label{optimalitycond1b}\\
\frac{\partial C}{\partial b}(a,b)&=\frac{m'(b)}{m^2(a,b)}\left[\int_a^b (\pi_1(b)-\pi_1(t))m'(t)dt - \frac{q_d+q_u}{S'(a)}\right]\label{optimalitycond2b}
\end{align}
from which the ordinary first order conditions \eqref{optimalitycond1} and \eqref{optimalitycond2} follow.
\end{proof}

\begin{remark}
It is at this point worth emphasizing that the findings of Lemma \ref{stationary} can be extended to the state-dependent cost setting
\begin{align}
\inf_{U,D\in \Lambda}\limsup_{T\rightarrow \infty}\frac{1}{T}\mathbb{E}_x\int_0^T\left(c(X_s^Z)ds + q_u(X_s^Z)\circ dU_s + q_d(X_s^Z)\circ dD_s\right)\label{JaZeProb}
\end{align}
originally considered in \cite{JaZe}. Focusing again on the reflecting barrier policies $U_t^a$ and $D_t^b$ yields
\begin{align*}
\lim_{T\rightarrow\infty}\frac{1}{T}\mathbb{E}_x\int_0^T\left(c(X_s^Z)ds + q_u(X_s^Z)\circ dU_s^a + q_d(X_s^Z)\circ dD_s^b\right)=
\frac{1}{m(a,b)}\left[\int_a^bc(t)m'(t)dt+\frac{q_u(a)}{S'(a)}+\frac{q_d(b)}{S'(b)}\right].
\end{align*}
Consequently, if the prices $q_u(x)$ and $q_d(x)$ are continuously differentiable, we find that if a pair $(a^\ast,b^\ast)$ of boundaries minimizing the expected long-run average cumulative costs exists, it has to satisfy in this case the ordinary first order conditions
\begin{align*}
\int_{a^\ast}^{b^\ast} (\tilde{\pi}_1(t)-\tilde{\pi}_1(b^\ast))m'(t)dt+\frac{q_d(a)+q_u(a)}{S'(a^\ast)} &=0, \\
\int_{a^\ast}^{b^\ast} (\tilde{\pi}_2(t)-\tilde{\pi}_2(a^\ast))m'(t)dt+\frac{q_d(b)+q_u(b)}{S'(b^\ast)} &=0,
\end{align*}
where $\tilde{\pi}_1(x)=c(x)+\frac{1}{2}\sigma^2(x)q_d'(x) + \mu(x)q_d(x)$ and $\tilde{\pi}_2(x)=c(x)-\frac{1}{2}\sigma^2(x)q_u'(x) - \mu(x)q_u(x)$.
\end{remark}

Before proceeding in the analysis of the considered control problem, we now connect our analysis to the free boundary value approach typically utilized in the determination of the optimal policy.
To this end, let $a < 0 < b$ be two arbitrary constant boundaries. Our objective is now to determine the twice continuously differentiable function $v:\mathbb{R}\mapsto \mathbb{R}_+$ as well as the two boundaries $a<0<b$ and the parameter $\lambda\in \mathbb{R}_+$ solving the free boundary problem
\begin{align}\label{freeboundary}
\begin{split}
(\mathcal{A}v)(x)+c(x)&= \lambda ,\quad x\in (a,b),\\
v(x) &=q_d(x-b)+v(b),\quad x\in [b,\infty),\\
v(x) &= q_u(a - x) + v(a),\quad x\in (-\infty,a].
\end{split}
\end{align}
 Since
$$
\frac{d}{d}\frac{v'(x)}{S'(x)}= (\mathcal{A}v)(x)m'(x) = (\lambda-c(x))m'(x)
$$
we find by integrating over the interval $(a,b)$ and invoking the boundary conditions $v'(a)=-q_u$ and $v'(b)=q_d$ that
\begin{align}\label{equilibrium}
\int_a^b c(t)m'(t)dt-\lambda m(a,b) + \frac{q_d}{S'(b)}+ \frac{q_u}{S'(a)}=0.
\end{align}
Consequently,
\begin{align}\label{suboptrepr}
\lambda = \frac{1}{m(a,b)}\left[\int_{a}^{b}c(t)m'(t)dt+\frac{q_u}{S'(a)}+\frac{q_d}{S'(b)}\right] = C(a,b).
\end{align}
Finally, imposing the conditions $v''(a)=v''(b)=0$ guaranteeing the second order differentiability of the value across the boundaries implies that
\begin{align}
\lambda = c(b)+q_d\mu(b) = c(a)-q_u\mu(a)\label{2ndorder}
\end{align}
which coincide with the ordinary first order optimality conditions \eqref{optimalitycond1} and \eqref{optimalitycond2}. Consequently,
we notice that the smoothness of the value is associated with the optimality of the boundaries.
We can now establish the following existence and uniqueness result.
\begin{theorem}\label{thm1}
The optimality conditions \eqref{optimalitycond1} and \eqref{optimalitycond2} have a uniquely determined solution
$(a^\ast, b^\ast)\in(-\infty,\hat{x}_2)\times(\hat{x}_1,\infty)$. Especially,
\begin{align}\label{representation}
C(a^\ast,b^\ast)= \lambda^\ast = \pi_2(a^\ast)=\pi_1(b^\ast).
\end{align}
Moreover, the marginal value $v'(x)$ satisfies the inequality
$-q_u\leq v'(x)\leq q_d$ for all $x\in \mathbb{R}$.
\end{theorem}
\begin{proof}
We first observe that since
$$
I_1(a,b)-I_2(a,b)=(\pi_2(a)-\pi_1(b))m(a,b),
$$
the roots are found from the set of thresholds satisfying $\pi_2(a)=\pi_1(b)$ as indicated by \eqref{2ndorder}. It is, therefore, sufficient to focus on the
function
\begin{align}
g(a) = \int_a^{b_a} c(t)m'(t)dt-\pi_1(b_a) m(a,b_a) + \frac{q_d}{S'(b_a)}+ \frac{q_u}{S'(a)}\label{roots}
\end{align}
where $b_a:=\{x\in[\hat{x}_1,\infty): \pi_1(x)=\pi_2(a)\}$ is defined for $a$ in appropriate subsets of $(-\infty,\hat{x}_2]$.

In order to establish the sets where $b_a$ is well-defined, consider first the case where $\hat{x}_1\geq \hat{x}_2$. If $\pi_1(\hat{x}_1)\geq \pi_2(\hat{x}_2)$ then our assumptions guarantee that there is a unique threshold $\hat{a}=\{x\in(-\infty,\hat{x}_2]:\pi_2(x) = \pi_1(\hat{x}_1)\}$ so that $b_a$ is well-defined for all $a\leq \hat{a}$.
Analogously, if $\pi_1(\hat{x}_1)\leq \pi_2(\hat{x}_2)$, then the function $b_a$ is well-defined for all $a\leq \hat{x}_2$. Consider now instead the case where $\hat{x}_1\leq \hat{x}_2$. It is then clear that our assumptions guarantee that if $\pi_1(\hat{x}_1)\geq \pi_2(\hat{x}_1)$ then we again observe that
$b_a$ is well-defined for all $a\leq \hat{a}$. If $\pi_2(\hat{x}_2)\geq \pi_1(\hat{x}_2)$, then  $b_a$ is well-defined for all $a\leq \hat{x}_2$. Finally, if either $\pi_2(\hat{x}_1)\geq \pi_1(\hat{x}_1)\geq \pi_2(\hat{x}_2)$ or $\pi_1(\hat{x}_2)\geq \pi_2(\hat{x}_2)\geq \pi_1(\hat{x}_1)$,
then there is a unique intersection point $\tilde{x}\in[\hat{x}_1,\hat{x}_2]$ at which $\pi_1(\tilde{x})=\pi_2(\tilde{x})$ and $b_a$ is well-defined for all $a\leq \tilde{x}$ and satisfies the condition $b_{\tilde{x}}=\tilde{x}$.

Having characterized the sets where $b_a$ is well-defined, we now plan to prove that $g(a)>0$ at the upper boundary of the set where $b_a$ is defined. Consider first the case where
$\pi_1(\hat{x}_1)\geq \pi_2(\hat{x}_2)$.
Utilizing \eqref{canonical} shows
\begin{align*}
g(\hat{a}) &=\int_{\hat{a}}^{\hat{x}_1} c(t)m'(t)dt-\pi_1(\hat{x}_1) m(\hat{a},\hat{x}_1) + \frac{q_d}{S'(\hat{x}_1)}+ \frac{q_u}{S'(\hat{a})}\\
&=\int_{\hat{a}}^{\hat{x}_1} \pi_1(t)m'(t)dt-\pi_1(\hat{x}_1) m(\hat{a},\hat{x}_1) +  \frac{q_d+q_u}{S'(\hat{a})}>0,
\end{align*}
since $\hat{x}_1=\argmin\{\pi_1(x)\}$. Consider now  the case where either $\hat{x}_1\geq \hat{x}_2$ and $\pi_1(\hat{x}_1)\leq \pi_2(\hat{x}_2)$ or $\hat{x}_1\leq \hat{x}_2$ and $\pi_1(\hat{x}_2)\leq \pi_2(\hat{x}_2)$. In those cases we find that
\begin{align*}
g(\hat{x}_2) &=\int_{\hat{x}_2}^{b_{\hat{x}_2}} c(t)m'(t)dt-\pi_1(b_{\hat{x}_2}) m(\hat{x}_2,b_{\hat{x}_2}) + \frac{q_d}{S'(b_{\hat{x}_2})}+ \frac{q_u}{S'(\hat{x}_2)}\\
&=\int_{\hat{x}_2}^{b_{\hat{x}_2}} \pi_2(t)m'(t)dt-\pi_2(\hat{x}_2) m(\hat{x}_2,b_{\hat{x}_2}) +  \frac{q_d+q_u}{S'(b_{\hat{x}_2})}>0,
\end{align*}
since $\hat{x}_2=\argmin\{\pi_2(x)\}$. Finally, if $\hat{x}_1\leq \hat{x}_2$ and either $\pi_2(\hat{x}_1)\geq \pi_1(\hat{x}_1)\geq \pi_2(\hat{x}_2)$ or $\pi_1(\hat{x}_2)\geq \pi_2(\hat{x}_2)\geq \pi_1(\hat{x}_1)$ holds, then
\begin{align*}
g(\tilde{x}) = \frac{q_d+q_u}{S'(\tilde{x})}>0
\end{align*}
proving the alleged positivity of $g(a)$ at the upper boundary of the set where $b_a$ is defined.

We now plan to establish that equation $g(a)=0$ has a unique root on $(-\infty,\hat{x}_2]$ by establishing that $g(a)$ is monotonically increasing on its domain and tends to $-\infty$ as $a\rightarrow -\infty$. To this end, assume that $a_1 < a_2$ and, therefore, that $b_1 > b_2$, where $b_i:=b_{a_i}, i=1,2$. Utilizing the definition of the function \eqref{roots}, as well as the identities \eqref{canonical} and $\pi_2(a_i)=\pi_1(b_i), i=1,2,$ yields
\begin{align*}
g(a_2)-g(a_1)&=\int_{a_1}^{a_2}(\pi_2(a_1)-\pi_2(t))m'(t)dt+\int_{b_2}^{b_1}(\pi_1(b_1)-\pi_1(t))m'(t)dt\\
&+(\pi_2(a_1)-\pi_2(a_2))m(a_2,b_2) > 0
\end{align*}
demonstrating that $g(a)$ is monotonically increasing. Moreover, since
\begin{align*}
g(a_2)-g(a_1) > (\pi_2(a_1)-\pi_2(a_2))m(a_2,b_2) \rightarrow \infty
\end{align*}
as $a_1\rightarrow-\infty$ we notice that $\lim_{a_1\rightarrow-\infty}g(a_1)=-\infty$. Combining this observation with the monotonicity and continuity of $g$ and the positivity of $g$ at the upper boundary of its domain demonstrates that equation $g(a)=0$ has a unique root $a^\ast\in (-\infty, \hat{x}_2]$. Moreover, since $g(a^\ast)=I_1(a^\ast, b^\ast)=I_2(a^\ast, b^\ast)=0$, where $b^\ast=b_{a^\ast}$,
we find that the pair $(a^\ast, b^\ast)$ constitutes the unique root of the optimality conditions \eqref{optimalitycond1} and \eqref{optimalitycond2}. Identity \eqref{representation}, in turn, follows directly from \eqref{suboptrepr}.

It remains to establish that the marginal value satisfies the inequality $-q_u\leq v'(x)\leq q_d$ for all $x\in \mathbb{R}$. It is clear that $v'(x)=q_d$ on $[b^\ast,\infty)$ and that $v'(x)=-q_u$ on $(-\infty,a^\ast]$. To establish that $-q_u\leq v'(x)\leq q_d$, on $(a^\ast,b^\ast)$, we first notice that since
$$
\frac{d}{dx}\left(\frac{v'(x)-q_d}{S'(x)}\right) = \left((\mathcal{A}v)(x)-q_d\mu(x)\right)m'(x)
$$
we obtain by invoking the boundary condition $v'(b^\ast)=q_d,$ and $\lambda^\ast=\pi_1(b^\ast)$ that
\begin{align}
\frac{v'(x)-q_d}{S'(x)} = \int_{x}^{b^\ast}(\pi_1(t)-\pi_1(b^\ast))m'(t)dt. \label{exp1}
\end{align}
We first prove that the integral expression in \eqref{exp1} is nonpositive on $[a^\ast,b^\ast]$. To see that this is indeed true, we first notice that if $\lim_{x\downarrow-\infty}\pi_1(x)\leq \pi_1(b^\ast)$, then the integrand is always nonpositive proving the alleged nonpositivity of the integral in that case. If, however,
$\lim_{x\downarrow-\infty}\pi_1(x)> \pi_1(b^\ast)$, then our assumptions on the function $\pi_1$ guarantee that there exists a uniquely defined state $y_1=\{x<\hat{x}_1: \pi_1(x)=\pi_1(b^\ast)\}$. However, since the integrand is nonpositive for all $t\in [y_1,b^\ast]$ and the integral is nonincreasing for $x\leq y_1$ we notice that the integral is nonpositive in that case as well. In order to complete the proof it is sufficient to show that
$$
\frac{v'(x)-q_d}{S'(x)} \geq -\frac{q_u+q_d}{S'(x)}
$$
for all $x\in[a^\ast,b^\ast]$. To this end consider now the function
\begin{align*}
D(x)=\int_{x}^{b^\ast}(\pi_1(t)-\pi_1(b^\ast))m'(t)dt+\frac{q_u+q_d}{S'(x)}.
\end{align*}
It is clear that $D(a^\ast)=0$, $D(b^\ast)=(q_u+q_d)/S'(b^\ast)>0$, and $D'(x)=(\pi_1(b^\ast)-\pi_2(x))m'(x)=(\pi_2(a^\ast)-\pi_2(x))m'(x)$. Two cases arise. If $\lim_{x\uparrow\infty}\pi_2(x)\leq \pi_2(a^\ast)$ then $D'(x)\geq0$ for all $x\in[a^\ast,b^\ast]$ and we are done. If, however, $\lim_{x\uparrow\infty}\pi_2(x)> \pi_2(a^\ast)$, then $D'(x)\geq 0$ for all $x\in[a^\ast, y_2]$ and $D'(x)< 0$ for all $x >y_2,$ where $y_2=\{x>\hat{x}_2: \pi_2(x)=\pi_2(a^\ast)\}$. Consequently, we notice that $D(x)\geq 0$ for all $x\in[a^\ast,b^\ast]$ in that case as well, completing the proof of our theorem.
\end{proof}

\begin{remark}
It is worth noticing that Theorem \ref{thm1} implies the marginal value can be re-expressed as a convex combination $$v'(x)= (1-p(x))q_d - p(x) q_u,$$
where
\begin{align}\label{marginalval}
p(x)=\frac{S'(x)}{q_u+q_d}\int_x^{b^\ast}(\pi_1(b^\ast)-\pi_1(t))m'(t)dt
\end{align}
satisfies $p(a^\ast)=1$ and $p(b^\ast)=0$.
\end{remark}

Theorem \ref{thm1} demonstrates that the monotonicity and limiting behavior of the functions $\pi_1$ and $\pi_2$ are the principal determinants of the pair of boundaries satisfying the first order optimality conditions \eqref{optimalitycond1} and \eqref{optimalitycond2}. According to Theorem \ref{thm1} the lower reflection boundary is situated on the set where $\pi_2$ is decreasing while the upper reflection boundary is situated on the set where $\pi_1$ is increasing. Consequently, no convexity or other second order properties are required for establishing the existence and uniqueness of a reflection. As is also clear from the proof or Theorem \ref{thm1} the assumed limiting behavior of the functions $\pi_1$ and $\pi_2$ are sufficient for the existence of the pair of boundaries. It is, however, clear that those assumptions are not necessary and can be weakened in some circumstances. More precisely, if the function $g$ defined in \eqref{roots} satisfies the condition $\lim_{a\rightarrow-\infty}g(a)<-\varepsilon$, where $\varepsilon>0$, then a unique pair $(a^\ast,b^\ast)$ satisfying the optimality conditions \eqref{optimalitycond1} and \eqref{optimalitycond2}
exists even when the limiting conditions $\lim_{x\uparrow\infty}\pi_1(x) = \lim_{x\downarrow-\infty}\pi_2(x)=\infty$ are not satisfied. We will illustrate this point later in an explicitly parameterized example.

Having proved the existence of the pair $(a^\ast,b^\ast)$ we now demonstrate that it is optimal as well and, therefore, that deviating from it results into higher long run average cumulative costs. Our main result on this subject is summarized in the next theorem.
\begin{theorem}\label{globalmin}
$(a^\ast, b^\ast)\in(-\infty,\hat{x}_2)\times(\hat{x}_1,\infty)$ is a global minimum on $C(a,b)$ and $C(a,b) \geq C(a^\ast, b^\ast)$ for all $-\infty < a < b<\infty$. Consequently, $\lambda \geq \lambda^\ast$ for all $-\infty < a < b<\infty$ and $Z_t^\ast=(U_t^{a^\ast}, D_t^{b^\ast})$ constitutes an optimal singular control within the considered class of reflection policies.
\end{theorem}
\begin{proof}
Let us first investigate the behavior of the function $I_1(a,b)$. If $b_2>b_1>a$, then
$$
I_1(a,b_1)-I_1(a,b_2)=(\pi_1(b_2)-\pi_1(b_1))m(a,b_1)+\int_{b_1}^{b_2}(\pi_1(b_2)-\pi_1(t))m'(t)dt.
$$
It is now clear from our assumptions on $\pi_1$ that $I_1(a,b)$ is increasing on $(-\infty,\hat{x}_1)$ and decreasing on $(\hat{x}_1,\infty)$ as a function of $b$. Moreover, if $b_2>b_1\geq \hat{x}_1$, then
$$
I_1(a,b_1)-I_1(a,b_2)>(\pi_1(b_2)-\pi_1(b_1))m(a,b_1)\rightarrow \infty
$$
as $b_2\rightarrow\infty$. Consequently, $\lim_{b\rightarrow\infty}I_1(a,b)=-\infty$ for all $a\in \mathbb{R}$. Combining these observation with identity $I_1(a,a) = (q_u+q_d)/S'(a)>0$ implies that
$I_1(a,b)=0$ has a unique root $\hat{b}_a$ for any $a\in \mathbb{R}$ and $I_1(a,b)\gtreqqless 0$ for $b \lesseqqgtr \hat{b}_a$.
Establishing now that $I_2(a,b)$ is increasing on $(-\infty,\hat{x}_2)$, decreasing on $(\hat{x}_2,\infty)$ as a function of $a$ and satisfies $\lim_{a\rightarrow-\infty}I_2(a,b)=-\infty$ for all $b\in \mathbb{R}$ is completely analogous. Consequently, we notice that $I_2(a,b)=0$ has a unique root $\hat{a}_b$ for any $b\in \mathbb{R}$ and $I_2(a,b)\lesseqqgtr 0$ for $a \lesseqqgtr \hat{a}_b$. Combining these observations with the uniqueness of the pair $(a^\ast,b^\ast)$ and the equations \eqref{optimalitycond1b} and \eqref{optimalitycond2b} completes the proof of the first claim of our theorem. The second statement then follows directly from Lemma \ref{stationary} and \eqref{suboptrepr}.
\end{proof}

Theorem \ref{globalmin} establishes that the pair $(a^\ast,b^\ast)$ satisfying the ordinary first order conditions constitutes the global minimum of the long run average cumulative costs. Since this minimum can be attained by utilizing an admissible local time reflection policy $Z_t^\ast=(U_t^{a^\ast}, D_t^{b^\ast})$ we notice that $Z_t^\ast$ indeed constitutes the optimal admissible policy.

It is naturally of interest to investigate how increased volatility affects the optimal policy and its value. Our main finding on that subject is summarized in the following.
\begin{lemma}
Assume that the value $v$ is convex. Then, increased volatility expands the continuation region and increases the expected long run average costs.
\end{lemma}
\begin{proof}
Denote as
$$
\hat{\mathcal{A}}=\frac{1}{2}\hat{\sigma}^2(x)\frac{d^2}{dx^2}+\mu(x)\frac{d}{dx}
$$
the differential operator associated with the more volatile process characterized by the diffusion coefficient $\hat{\sigma}$ satisfying the inequality $\hat{\sigma}(x)>\sigma(x)$ for all
$x\in \mathbb{R}$. If $v$ is convex, then clearly
$$
(\hat{\mathcal{A}}v)(x)+c(x)-\lambda^\ast = \frac{1}{2}(\hat{\sigma}^2(x)-\sigma^2(x))v''(x)>0
$$
implying that
$$
\frac{v'(b)}{\hat{S}'(b)}-\frac{v'(a)}{\hat{S}'(a)}>\lambda^\ast \hat{m}(a,b) - \int_a^bc(t)\hat{m}'(t)dt
$$
where $\hat{S}'$ denotes the density of the scale and $\hat{m}'$ the speed density of the more volatile underlying. Since $-q_u \leq v'(x) \leq q_d$ we notice that
$$
\lambda^\ast < \frac{1}{\hat{m}(a,b)}\left(\int_a^bc(t)\hat{m}'(t)dt + \frac{q_d}{\hat{S}'(b)}+\frac{q_u}{\hat{S}'(a)}\right).
$$
Since this inequality is valid for all $a,b$ it is valid for the optimal ones as well and, consequently we have $\lambda^\ast < \hat{\lambda}$, where $\hat{\lambda}$ denotes the the expected long run average costs in the more volatile setting. Since $\hat{\lambda}=\pi_1(\check{b})=\pi_2(\check{a})$, where $\check{a},\check{b}$ denote the optimal boundaries in the more volatile setting, we notice by combining this observation with the assumed monotonicity of the mappings $\pi_1$ and $\pi_2$ that $\check{a} < a^\ast$ and $\check{b}>b^\ast$, thus completing the proof of our lemma.
\end{proof}

In many practical economic and financial applications of singular stochastic control theory only one control policy can be endogenously determined while the other is exogenously set through the constraints affecting decision making. For example, cash flow policies may be subject to compulsory recapitalization should the prevailing reserves fall below an exogenously set critical level. In a completely analogous manner, reserve accumulation policies may be subject to obligatory redistribution should the reserves exceed a preset level. As intuitively is clear, in such situations, the considered problems constitute special cases of the general problem considered in Theorem \ref{thm1}. Our main findings focusing on these type of problems are now summarized in the next corollary of Theorem \ref{thm1}.
\begin{corollary}
(A) Assume that the lower boundary $a\in \mathbb{R}$ is exogenously set. Then, there exists a unique optimal reflection boundary $b^\ast_a\in (\hat{x}_1,\infty)$ satisfying the first order condition
\begin{align}
\int_a^{b^\ast_a}(\pi_1(b^\ast_a)-\pi_2(t))m'(t)dt=\frac{q_u+q_d}{S'(b^\ast_a)}.\label{singleb}
\end{align}
Moreover, $v'(x)\leq q_d$ for all $x\in \mathbb{R}$.\\
\noindent(B) Assume that the upper boundary $b\in \mathbb{R}$ is exogenously set. Then, there exists a unique optimal reflection boundary $a^\ast_b\in (-\infty,\hat{x}_2)$ satisfying the first order condition
\begin{align}
\int_{a^\ast_b}^{b}(\pi_2(a^\ast_b)-\pi_1(t))m'(t)dt=\frac{q_u+q_d}{S'(a^\ast_b)}.\label{singlea}
\end{align}
Moreover, $v'(x)\geq -q_u$ for all $x\in \mathbb{R}$.\\
\end{corollary}
\begin{proof}
(A) Utilizing \eqref{equilibrium} in addition to the optimality condition $v''(b^\ast_a)=0$ implies that equation \eqref{singleb} has to be satisfied. In order to establish the existence and uniqueness of $b^\ast$, consider the function
$$
F(b)=\int_a^{b}(\pi_1(b)-\pi_2(t))m'(t)dt-\frac{q_u+q_d}{S'(b)}.
$$
It is clear that $F(a)=-(q_u+q_d)/S'(a)<0$ and $F'(b)=\pi_1'(b)m(a,b)>0$ for $b > \hat{x}_1\vee a$. If $b>y>\hat{x}_1\vee a$ then
$$
F(b)-F(y) = \int_y^b\pi_1'(t)m(a,t)dt > m(a,y)(\pi_1(b)-\pi_1(y))\rightarrow \infty
$$
as $b\uparrow\infty$. Comining this observation with the monotonicity of $\pi_1$ proves the existence and uniqueness of the threshold $b^\ast.$ The validity of inequality $v'(x)\leq q_d$ for all $x\in \mathbb{R}$ follows from \eqref{exp1} as proved in Theorem \ref{thm1}. Establishing part (B) is completely analogous.
\end{proof}

As is clear from Theorem \ref{thm1}, the potential asymmetry of the costs $c$ affects the optimal reflection boundaries only through its effect on the functions $\pi_1$ and $\pi_2$ which depend also on the drift coefficient of the underlying dynamics. The role of the costs is naturally pronounced in the absence of a drift since in that setting $\pi_1\equiv \pi_2$ and the incentives to exert the control policy are solely determined by the costs and volatility. Our findings on this special case are now summarized in the following.
\begin{theorem}\label{thm2}
Assume that $\mu\equiv 0$. Then, the optimality conditions
\begin{align}
I_1(a,b) &= \frac{1}{2}(q_d+q_u) + \int_a^b \frac{c(t)}{\sigma^2(t)}dt - c(b)\int_a^b \frac{1}{\sigma^2(t)}dt=0\label{optimalitycondnodrift1}\\
I_2(a,b) &= \frac{1}{2}(q_d+q_u) + \int_a^b \frac{c(t)}{\sigma^2(t)}dt - c(a)\int_a^b \frac{1}{\sigma^2(t)}dt=0\label{optimalitycondnodrift2}
\end{align}
have a uniquely determined solution $(a^\ast, b^\ast)\in(-\infty,0)\times(0,\infty)$ such that $\lambda^\ast=c(a^\ast)=c(b^\ast)$.
Moreover,
\begin{itemize}
  \item[(i)] The value $v$ is strictly convex and satisfies the inequality $-q_u \leq v'(x)\leq q_d$ for all $x\in \mathbb{R}$.
  \item[(ii)] Increased volatility expands the continuation region by decreasing the lower boundary and increasing the upper boundary.
\end{itemize}
\end{theorem}
\begin{proof}
The first alleged result is a direct implication of Theorem \ref{thm1}. In order to establish the convexity of the value $v$, we first notice that $\lambda^\ast=c(a^\ast)=c(b^\ast)>c(x)$ for all $x\in(a^\ast,b^\ast)$ since $c(x)$ is decreasing on $\mathbb{R}_-$ and increasing on $\mathbb{R}_+$. Since $v\in C^2(\mathbb{R})$ is linear on $(-\infty,a^\ast)\cup(b^\ast,\infty)$ and satisfies on $(a^\ast,b^\ast)$ the inequality
$$
\frac{1}{2}\sigma^2(x)v''(x)=\lambda^\ast-c(x) >0
$$
then proves that $v$ is convex on $\mathbb{R}$. This observation also implies that $-q_u \leq v'(x)\leq q_d$ for all $x\in \mathbb{R}$.
It remains to prove that increased volatility expands the continuation region. To see that this is the case, we first observe that
$c(t) < c(b_a) = c(a)$ for all $t\in (a, b_a)$, where $b_a=\{x>0: c(x)=c(a)\}$ for all $a<0$. Consequently, if $\hat{\sigma}^2(x)>\sigma^2(x)$ for
all $x\in \mathbb{R}$ then
$$
g(a)=\int_a^{b_a}\frac{c(t) - c(a)}{\sigma^2(t)}dt+\frac{1}{2}(q_d+q_u) < \int_a^{b_a}\frac{c(t) - c(a)}{\hat{\sigma}^2(t)}dt+\frac{1}{2}(q_d+q_u)=:\hat{g}(a)
$$
for all $a<0$. Since $g$ and $\hat{g}$ are increasing on $\mathbb{R}_-$ we notice that $0=g(a^\ast) < \hat{g}(a^\ast)$ implying that the root $\tilde{a}$
of equation $\hat{g}(a)=0$ is below $a^\ast$. Consequently, $b_{\tilde{a}}>b_{a^\ast}=b^\ast$ completing the proof of the alleged claim.
\end{proof}

Theorem \ref{thm2} characterizes the optimal boundaries and the value explicitly in the case where the controlled diffusion constitutes a time changed Brownian motion.
Interestingly, our findings indicate that the familiar conclusion on the negative effect of increased volatility on the incentives to utilize the irreversible control policy
are true also in this case. The higher volatility gets, the larger the continuation region becomes. An interesting direct implication of Theorem \ref{thm2} is now summarized in our next corollary.
\begin{corollary}
Assume that $\mu\equiv 0$ and that the cost function $c(x)$ is even. Then $a^\ast=-b^\ast$ where the optimal boundary $b^\ast$ constitutes the unique positive root of equation
\begin{align}\label{optimalitycondnodrifteven}
\hat{g}(b^\ast)=\frac{1}{2}(q_d+q_u)+\int_{-b^\ast}^{b^\ast} \frac{c(t)}{\sigma^2(t)}dt - c(b^\ast)\int_{-b^\ast}^{b^\ast} \frac{1}{\sigma^2(t)}dt=0.
\end{align}
\end{corollary}
\begin{proof}
The alleged result is a direct implication of  Theorem \ref{thm2}.
\end{proof}

It is worth emphasizing that the optimality conditions \eqref{optimalitycond1} and \eqref{optimalitycond2} can also be re-expressed in terms of the stationary distribution of the controlled dynamics. More precisely,  since the process is maintained between the boundaries $a$ and $b$ and $m(a,b)<\infty$ when $-\infty < a < b < \infty$ we notice that $X_t^Z\rightarrow \bar{X}_\infty$ where $\bar{X}_\infty$ is distributed according to the stationary probability measure characterized by its density (cf. \cite{BS}, pp. 36--38)
$$
p_{a,b}(x) = \frac{m'(x)}{m(a,b)}.
$$
Utilizing this definition shows that the optimality conditions \eqref{optimalitycond1} and \eqref{optimalitycond2} can be re-expressed as
\begin{align}
\mathbb{E}\left[\pi_1(\bar{X}_\infty)\right] - \pi_1(b) &=-\frac{q_d+q_u}{S'(a)m(a,b)}\label{stationarycond1}\\
\mathbb{E}\left[\pi_2(\bar{X}_\infty)\right] - \pi_2(a) &=-\frac{q_d+q_u}{S'(b)m(a,b)}.\label{stationarycond2}
\end{align}
As we will later see in the analysis of single boundary control problems, the left hand side of these identities appears in the characterization of the optimal boundaries within a one-sided
reflection setting. In those cases the right hand side of the optimality conditions equals zero. Consequently, we notice that in the two-boundary setting the difference between the expected long-run cost flow and its value at the boundary falls short its value in the single boundary setting. As intuitively is clear, this difference is based on the larger flexibility a decision maker has in the present setting where the underlying can be controlled both upwards as well as downwards. Since this flexibility is partially lost in the single boundary setting, the cost savings are naturally lower in that case. Our key observation on the stationary characterization of the two-boundary problem is now summarized in the following theorem how identity \eqref{suboptrepr} arises by applying a standard reflection policy.

\section{Explicit Illustrations}

\subsection{Example A: Controlled Brownian Motion with Drift}

We now consider as an example the stationary singular stochastic control problem of the process
\begin{align}
dX^{Z}_t = \mu dt + \sigma dW_t + dU_t -
dD_t, {\hskip .1in}X^{Z}_0 = x\in \mathbb{R},\label{sdeBMD}
\end{align}
where $\mu>0$ and $\sigma>0$. In order to investigate how the potential asymmetries in the costs affect the optimal policy, we assume that $q_d=q_u=1$ and consider the piecewise linear cost $c(x) = \max(-k_1 x,k_2 x)$, where $k_1,k_2$ are known exogenously determined positive parameters. It is clear that $\pi_1(x)=\max(-k_1 x,k_2 x)+\mu$ and $\pi_2(x)=\max(-k_1 x,k_2 x)-\mu$ satisfy the required monotonicity and limiting conditions and consequently, our results apply. Standard integration yields
\begin{align}
I_1(a,b) &=\frac{e^{\frac{2\mu a}{\sigma^2}} \left(2 \mu (k_1  a + k_2 b +2\mu) - k_1 \sigma^2\right) +\sigma^2
   \left(k_1+k_2-k_2 e^{\frac{2 \mu b}{\sigma^2}}\right)}{2 \mu^2}\\
I_2(a,b) &= \frac{e^{\frac{2 \mu b}{\sigma^2}} \left(2 \mu (k_1  a + k_2 b +2\mu) -k_2 \sigma^2\right)+
\sigma^2 \left(k_1+k_2 - k_1 e^{\frac{2 \mu a}{\sigma^2}}\right)}{2 \mu ^2}.
\end{align}
Now
$$
b_a=-\frac{k_1}{k_2}a-\frac{2\mu}{k_2}
$$
implying that
$$
g(a) = \frac{\sigma^2}{2 \mu^2}\left(k_1+k_2-k_2 e^{-\frac{2 \mu}{\sigma ^2}\left(\frac{k_1}{k_2}a+\frac{2\mu}{k_2}\right)}-k_1 e^{\frac{2 \mu a}{\sigma^2}}\right).
$$
In this case $\hat{a} = -2\mu/k_1,$ $\lim_{a\downarrow -\infty} g(a)=-\infty$,
\begin{align*}
g(\hat{a}) =  \frac{\sigma^2 k_1}{2 \mu^2}\left(1-e^{-\frac{4 \mu^2}{\sigma^2 k_1}}\right) > 0,\\
\end{align*}
and
$$
g'(a) = \frac{k_1}{\mu}e^{\frac{2\mu a}{\sigma^2}}\left(e^{-\frac{2\mu a}{\sigma^2}\left(1+\frac{k_1}{k_2}\right)-\frac{4\mu^2}{\sigma^2 k_1}}-1\right)>0
$$
for all $a<-2\mu/(k_1+k_2)$. Since, $\hat{a}<-2\mu/(k_1+k_2)$ we notice that equation $g(a)=0$ has a unique root $a^\ast\in(-\infty,\hat{a})$. The case where $\mu\equiv 0$ is of special interest, since in that case
$$
b_a=-\frac{k_1}{k_2}a
$$
and
$$
g(a) = 1-\frac{k_1}{2\sigma^2}\left(1+\frac{k_1}{k_2}\right)a^2.
$$
Consequently, in the absence of drift the optimal boundaries read as
\begin{align*}
a^\ast &= -\sqrt{\frac{2k_2\sigma^2}{k_1(k_1+k_2)}}\\
b^\ast &= \sqrt{\frac{2k_1\sigma^2}{k_2(k_1+k_2)}}.
\end{align*}
We notice directly that increased volatility expands the continuation region. Moreover, since $b^\ast/a^\ast = -k_1/k_2$ we notice that the relative distance between the two optimal thresholds is inversely proportional to the marginal cost ratio $k_1/k_2$. The elasticities of these thresholds with respect to changes in the marginal costs read as
\begin{align*}
\frac{k_1}{a^\ast}\frac{\partial a^\ast}{\partial k_1} &= -\frac{2\theta+1}{2(\theta+1)}\quad \frac{k_1}{b^\ast}\frac{\partial b^\ast}{\partial k_1} =\frac{1}{2(\theta+1)}\\
\frac{k_2}{a^\ast}\frac{\partial a^\ast}{\partial k_2} &=\frac{\theta}{2(\theta+1)}\quad \quad\frac{k_2}{b^\ast}\frac{\partial b^\ast}{\partial k_2} =-\frac{\theta+2}{2(\theta+1)}
\end{align*}
where $\theta=k_1/k_2$. Hence, we notice that the sensitivities of the thresholds with respect to changes in the marginal cost ratio are asymmetric as well.

It is worth noticing that in the symmetric setting where $k_1= k_2 = k>0$ the function $g$ reads as
$$
g(a) = \frac{k \sigma^2}{2 \mu^2}e^{-\frac{2 \mu a}{\sigma^2}}\left(2e^{\frac{2 \mu a}{\sigma^2}}- e^{-\frac{4 \mu^2}{k\sigma^2}}-e^{\frac{4 \mu a}{\sigma^2}}\right)
$$
from which we can directly deduce that the optimal thresholds are
\begin{align*}
a^\ast &= \frac{\sigma^2}{2\mu}\ln\left(1-\sqrt{1-e^{-\frac{4 \mu^2}{k\sigma^2}}}\right)\\
b_{a^\ast} &=-\frac{\sigma^2}{2\mu}\ln\left(1-\sqrt{1-e^{-\frac{4 \mu^2}{k\sigma^2}}}\right)-\frac{2\mu}{k}.
\end{align*}
Invoking L'Hospital's rule in the symmetric setting yields
\begin{align*}
\lim_{\mu\rightarrow 0}a^\ast &= -\frac{\sigma}{\sqrt{k}}\\
\lim_{\mu\rightarrow 0}b_{a^\ast} &=\frac{\sigma}{\sqrt{k}}
\end{align*}
demonstrating how increased volatility increases linearly the continuation region in that case.

\subsection{Example B: Controlled Ornstein-Uhlenbeck-process}
In order to illustrate our findings in a mean-reverting setting, we now consider the stationary singular stochastic control problem of the process
\begin{align}
dX^{Z}_t = (\alpha - \beta X_t^{Z})dt + \sigma dW_t + dU_t -
dD_t, {\hskip .1in}X^{Z}_0 = x\in \mathbb{R},\label{sdeOU}
\end{align}
where $\alpha>0,\beta>0,$ and $\sigma>0$. We again assume that $q_d=q_u=1$ and that the cost is of the piecewise linear form $c(x) = \max(-k_1 x,k_2 x)$, where $k_1,k_2$ are known exogenously determined positive parameters satisfying the inequality $\beta<\min(k_1,k_2)$. It is clear that now $\pi_1(x)=\max(-k_1 x,k_2 x)+\alpha - \beta x$ and $\pi_2(x)=\max(-k_1 x,k_2 x)-\alpha + \beta x$ satisfy the required monotonicity and limiting conditions and consequently, our results again apply. In this case
\begin{align*}
S'(x) &= e^{-\frac{2\alpha x}{\sigma^2}+\frac{\beta}{\sigma^2}x^2}\\
m(a,b) &= \frac{2}{\sigma}\sqrt{\frac{\pi}{\beta}} e^{\frac{\alpha^2}{\beta\sigma^2}}\left(\Phi\left(\frac{\sqrt{2\beta}b}{\sigma}-\frac{\sqrt{2} \alpha }{\sqrt{\beta}\sigma}\right)-\Phi\left(\frac{\sqrt{2\beta}a}{\sigma}-\frac{\sqrt{2} \alpha }{\sqrt{\beta}\sigma}\right)\right),
\end{align*}
and
\begin{align*}
\textstyle
\int_a^bc(t)m'(t)dt & \scriptstyle= \frac{k_2}{\beta}e^{\frac{\alpha^2}{\sigma^2\beta}}\sqrt{2\pi}\left(\Phi'\left(\frac{\sqrt{2}\alpha}{\sqrt{\beta}\sigma}\right)-
\Phi'\left(\frac{\sqrt{2\beta} b}{\sigma}-\frac{\sqrt{2}\alpha}{\sqrt{\beta}\sigma}\right)+\frac{\sqrt{2}\alpha}{\sqrt{\beta}\sigma}\left(\Phi\left(\frac{\sqrt{2\beta} b}{\sigma}-\frac{\sqrt{2}\alpha}{\sqrt{\beta}\sigma}\right)-\Phi\left(-\frac{\sqrt{2}\alpha}{\sqrt{\beta}\sigma}\right)\right)\right)\\
&\scriptstyle-\frac{k_1}{\beta}e^{\frac{\alpha^2}{\sigma^2\beta}}\sqrt{2\pi}\left(\Phi'\left(\frac{\sqrt{2\beta} a}{\sigma}-\frac{\sqrt{2}\alpha}{\sqrt{\beta}\sigma}\right)-\Phi'\left(\frac{\sqrt{2}\alpha}{\sqrt{\beta}\sigma}\right)+\frac{\sqrt{2}\alpha}{\sqrt{\beta}\sigma}
\left(\Phi\left(-\frac{\sqrt{2}\alpha}{\sqrt{\beta}\sigma}\right)-\Phi\left(\frac{\sqrt{2\beta} a}{\sigma}-\frac{\sqrt{2}\alpha}{\sqrt{\beta}\sigma}\right)\right)\right),
\end{align*}
where $\Phi(x)$ denotes the cumulative distribution function of a standard normal random variable. As is clear from the expressions above, the auxiliary function
$$g(a)=I_1\left(a,\frac{\beta -k_1}{k_2-\beta }a-\frac{2 \alpha }{k_2-\beta }\right)$$ takes in this case a relatively complex form and is, thus, left unstated. The boundaries of the continuation region are illustrated in Figure \ref{OUh} as functions of the marginal cost $k_2$ for two volatilities ($\sigma=0.5, 0.75$) under the assumptions that
$\alpha = 0.1, \beta= 0.1$, and $k_1 = 0.5$. As is clear from the figure, an increase in the marginal cost $k_2$ associated with controlling the underlying diffusion downwards decreases both boundaries. The reason for this observation is intuitively clear. As $k_2$ becomes larger, controlling the diffusion downwards becomes more expensive while controlling the diffusion upwards becomes relatively cheaper. In line with standard real option models of irreversible decision making our numerical illustration indicates that increased volatility expands the continuation region by rising $b^\ast$ and lowering $a^\ast$. Interestingly, the relative impact of increased volatility appears to be stronger with respect to the boundary associated with the control policy with a lower marginal cost. This shows the intricate interaction between volatility and costs in the stationary singular stochastic control setting.
\begin{figure}[!ht]
\begin{center}
\includegraphics[width=0.5\textwidth]{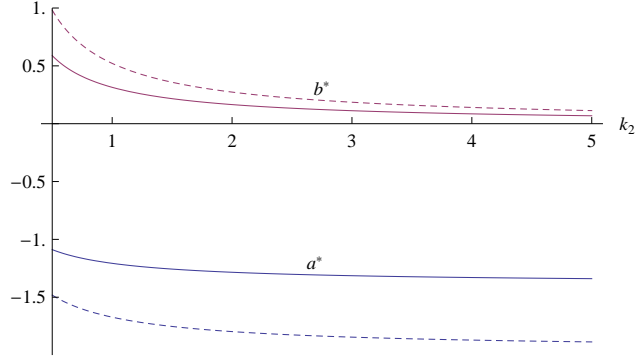}
\end{center}
\caption{\small The Optimal Reflection Boundaries ($\sigma=0.5$ uniform, $\sigma=0.75$ dashed)}\label{OUh}
\end{figure}

\subsection{Example C: A General Symmetric Problem}

In order to analyze an example leading to a symmetric control policies (in the spirit of the seminal study by \cite{Kar1}) we now consider the stationary singular stochastic control problem of the process
\begin{align}
dX^{Z}_t = \mu(X_t^{Z})dt + \sigma(X_t^{Z}) dW_t + dU_t -
dD_t, {\hskip .1in}X^{Z}_0 = x\in \mathbb{R},\label{sdesymmetric}
\end{align}
where the drift coefficient $\mu:\mathbb{R}\mapsto \mathbb{R}$ is assumed to be an odd function. In order to analyze symmetric stationary singular stochastic control policies, we assume that both the cost function $c(x)$ as well as the volatility coefficient $\sigma(x)$ are even and that $q_d=q_u=1$. It is clear that in this case for all $a<0$ it holds $$\pi_2(a) = c(a) - \mu(a) = c(-a)+\mu(-a) = \pi_1(-a)$$ indicating that the optimal policy is symmetric. By invoking symmetry we find that the upper optimal reflection boundary $b^\ast$ constitutes the unique root of equation $h(b)=0$, where
$$
h(b) = \int_0^b e^{\int_0^t\frac{2\mu(s)}{\sigma^2(s)}ds}\frac{(c(t)+\mu(t))}{\sigma^2(t)}dt-(c(b)+\mu(b))\int_0^b \frac{1}{\sigma^2(t)}e^{\int_0^t\frac{2\mu(s)}{\sigma^2(s)}ds}dt + \frac{1}{2}.
$$

We illustrate the boundaries of the continuation region in the case where $\mu(x) = \mu x, \sigma(x)=\sigma>0$, and $c(x)=|x|$. It is clear that in this case
$$
F(b)=\frac{1+\mu}{2\mu}\left(e^{\mu\left(\frac{b}{\sigma}\right)^2}-1\right)-(1+\mu)\left(\frac{b}{\sigma}\right)\int_0^{\frac{b}{\sigma}}e^{\mu y^2}dy+\frac{1}{2}
$$
proving that $b^\ast = \kappa^\ast \sigma$ and that $a^\ast = -\kappa^\ast \sigma$, where $\kappa^\ast>0$ constitutes the unique positive root of equation
$$
\frac{1+\mu}{2\mu}\left(e^{\mu{\kappa^\ast}^2}-1\right)+\frac{1}{2}-(1+\mu)\kappa^\ast \int_0^{\kappa^\ast}e^{\mu y^2}dy=0.
$$
The boundaries of the continuation region are illustrated in Figure \ref{Linearh} as functions of volatility under the assumption that
$\mu = 0.05$ (implying that $\kappa^\ast\approx 0.972044$).
\begin{figure}[!ht]
\begin{center}
\includegraphics[width=0.5\textwidth]{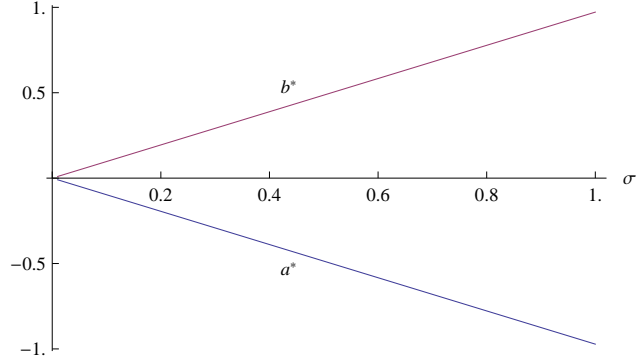}
\end{center}
\caption{\small The Optimal Reflection Boundaries}\label{Linearh}
\end{figure}

It is at this point worth emphasizing that as was argued after Theorem \ref{thm1}, the existence of the optimal reflection boundaries can be guaranteed even in cases where the limiting conditions on the mappings $\pi_1$ and $\pi_2$ are not satisfied. It is clear that in the present example the existence and uniqueness is guaranteed provided that $\lim_{b\rightarrow\infty}h(b) < -\varepsilon$, where $\varepsilon>0$. To illustrate such a case, assume that $c(x)=\max(1-e^x,1-e^{-x})$, $\mu\equiv 0$, and $\sigma(x)=\sigma>0$. In that case
$$
h(b)=\frac{1}{2}+\frac{1}{\sigma^2}\left(e^{-b}(1+b)-1\right)
$$
proving that the optimal boundary $b^\ast$ exists provided that $\sigma < \sqrt{2}$.

\begin{remark}
The example above indicates that the assumed continuity of the drift coefficient is not necessary for the validity of our results in some circumstances since the controlled stochastic differential equation \eqref{sde} has a weak solution also in situations where the drift coefficient is bounded but not necessarily continuous. Especially, if the underlying follows a Brownian motion with alternating drift
$$
dX^{Z}_t = \mu \sgn(X^{Z}_t) dt + \sigma X^{Z}_t dW_t + dU_t -
dD_t, {\hskip .1in}X^{Z}_0 = x\in \mathbb{R},
$$
$c(x)$ is an even function, and $q_d=q_u=q>0$, then there exists a unique pair of optimal reflection boundaries $-\infty < z^\ast < y^\ast < \infty$ satisfying the symmetry condition $y^\ast = -z^\ast$, where $y^\ast$ constitutes the unique positive root of equation
$$
q+\frac{2}{\sigma^2}\int_{0}^{y^\ast}e^{\frac{2\mu}{\sigma^2}t}c(t)dt-\left(e^{\frac{2\mu}{\sigma^2}y^\ast}-1\right)\frac{c(y^\ast)}{\mu}=0.
$$
\end{remark}

\section{Optimal Reflection at a Single Boundary}

In many practical applications focusing on the complete irreversibility of the implemented policy the control policy can take the controlled stock only to a single direction. Examples of such actions are, for example, rational harvesting planning or optimal dividend distribution. For the sake of completeness we will shortly focus on that class of problems in this section. More precisely, we will investigate the determination of the admissible policies $D$ and $U$ attaining the minimal expected asymptotic values
\begin{align}
\lim_{T\rightarrow \infty}\frac{1}{T}\mathbb{E}_x\int_0^T\left(c(X_s^Z)ds + q_d dD_s\right)\label{sscpdown}
\end{align}
and
\begin{align}
\lim_{T\rightarrow \infty}\frac{1}{T}\mathbb{E}_x\int_0^T\left(c(X_s^Z)ds + q_u dU_s\right),\label{sscpup}
\end{align}
respectively. As is clear, our original assumptions in the two-boundary setting need to be adjusted to the problems at hand. In the case of the downward control problem \eqref{sscpdown} we assume in addition to our assumptions on $\pi_1$ the following:
\begin{itemize}
  \item[(D1)] $\lim_{x\downarrow -\infty}m(x,b)<\infty$ and $\lim_{x\downarrow -\infty}\int_x^b c(t)m'(t)dt<\infty$ for all $b\in \mathbb{R}$
  \item[(D2)] $\lim_{x\downarrow -\infty}S'(x)=\infty$.
\end{itemize}
On the other hand, in the case of the upward control problem \eqref{sscpup} we assume in addition to our assumptions on $\pi_2$ the following:
\begin{itemize}
  \item[(U1)] $\lim_{x\uparrow \infty}m(a,x)<\infty$ and $\lim_{x\uparrow \infty}\int_a^x c(t)m'(t)dt<\infty$ for all $a\in \mathbb{R}$
  \item[(U2)] $\lim_{x\uparrow \infty}S'(x)=\infty$.
\end{itemize}
It is at this point worth mentioning that assumptions (D1) and (D2) guarantee that the lower boundary $-\infty$ is either natural or entrance (and, hence, {\em unattainable}) for the controlled diffusion in the absence
of interventions. An analogous argument holds for the upper boundary $\infty$ under the assumptions (U1) and (U2).

Before establishing our principal existence and uniqueness results on the optimal policy, we first notice that utilizing Lemma \ref{stationary} in addition to \eqref{canonical} and the conditions (D1), (D2), (U1), and (U2) to the problems at hand
yields in the downward reflection case
\begin{align*}
\lim_{T\rightarrow\infty}\frac{1}{T}\mathbb{E}_x\int_0^T\left(c(X_s^Z)ds + q_d dD^{b}_{s}\right)=
\frac{1}{m(-\infty,b)}\int_{-\infty}^b\pi_1(t)m'(t)dt
\end{align*}
for an arbitrary boundary $b\in \mathbb{R}$ and in the upward reflection case
\begin{align*}
\lim_{T\rightarrow\infty}\frac{1}{T}\mathbb{E}_x\int_0^T\left(c(X_s^Z)ds + q_U dU^{a}_{s}\right)=
\frac{1}{m(a,\infty)}\int_{a}^{\infty}\pi_2(t)m'(t)dt
\end{align*}
for an arbitrary boundary $a\in \mathbb{R}$. Define the functions $J_1:\mathbb{R}\mapsto \mathbb{R}$ and $J_2:\mathbb{R}\mapsto \mathbb{R}$ as
\begin{align*}
J_1(b) &=\frac{1}{m(-\infty,b)}\int_{-\infty}^b\pi_1(t)m'(t)dt\\
J_2(a) &=\frac{1}{m(a,\infty)}\int_{a}^{\infty}\pi_2(t)m'(t)dt.
\end{align*}
Standard differentiation then yields
\begin{align*}
J_1'(b) &=\frac{m'(b)}{m^2(-\infty,b)}\int_{-\infty}^b(\pi_1(b)-\pi_1(t))m'(t)dt\\
J_2'(a) &=\frac{m'(a)}{m^2(a,\infty)}\int_{a}^{\infty}(\pi_2(t)-\pi_2(a))m'(t)dt.
\end{align*}
We can now immediately establish the following result.
\begin{theorem}\label{onebdry}
(A) The optimality condition
\begin{align}
\int_{-\infty}^{b^\ast}\pi_1(t)m'(t)dt=\pi_1(b^\ast)m(-\infty,b^\ast)\label{upper}
\end{align}
has a unique solution $b^\ast\in (\hat{x}_1,\infty)$ such that
$$
\lim_{T\rightarrow\infty}\frac{1}{T}\mathbb{E}_x\int_0^T\left(c(X_s^Z)ds + q_d dD^{b^\ast}_{s}\right)=\pi_1(b^\ast).
$$
Moreover,
$$
\lim_{T\rightarrow\infty}\frac{1}{T}\mathbb{E}_x\int_0^T\left(c(X_s^Z)ds + q_d dD^{b}_{s}\right) \geq \pi_1(b^\ast)
$$
for all $b\neq b^\ast$.\\

\noindent(B) The optimality condition
\begin{align}
\int_{a^\ast}^{\infty}\pi_2(t)m'(t)dt=\pi_2(a^\ast)m(a^\ast,\infty)\label{lower}
\end{align}
has a unique solution $a^\ast\in (-\infty,\hat{x}_2)$ such that
\begin{align*}
\lim_{T\rightarrow\infty}\frac{1}{T}\mathbb{E}_x\int_0^T\left(c(X_s^Z)ds + q_U dU^{a^\ast}_{s}\right)= \pi_2(a^\ast).
\end{align*}
Moreover,
\begin{align*}
\lim_{T\rightarrow\infty}\frac{1}{T}\mathbb{E}_x\int_0^T\left(c(X_s^Z)ds + q_U dU^{a}_{s}\right)\geq \pi_2(a^\ast)
\end{align*}
for all $a\neq a^\ast$.
\end{theorem}
\begin{proof}
(A) Consider the function
$$
l_1(b) = \int_{-\infty}^b(\pi_1(b)-\pi_1(t))m'(t)dt.
$$
It is clear that $l_1(b)<0$ for all $b\leq \hat{x}_1$ and that $l_1'(b)=\pi_1'(b)m(-\infty,b)>0$ for all $b>\hat{x}_1$. If $b> z > \hat{x}_1$ we notice that
$$
l_1(b) - l_1(z) = \int_z^b\pi_1'(t)m(-\infty,t)dt > m(-\infty,z)(\pi_1(b)-\pi_1(z))\rightarrow\infty
$$
as $b\uparrow\infty$. Hence, equation $l_1(b)=0$ has a unique root $b^\ast = \argmin\{J_1(b)\}>\hat{x}_1$ constituting the global minimum of the functional $J_1(b)$ and satisfying
$J_1(b^\ast)=\pi_1(b^\ast)$. Establishing part (B) is entirely analogous.
\end{proof}

Theorem \ref{onebdry} summarizes our main findings on the single boundary setting. It is worth noticing that the optimality condition \eqref{upper} can be
re-expressed as
$$
\mathbb{E}[\pi_1(\bar{X}_\infty)]-\pi_1(b^\ast) = 0.
$$
Analogously, the optimality condition \eqref{lower} can be
re-expressed as
$$
\mathbb{E}[\pi_2(\bar{X}_\infty)]-\pi_2(a^\ast)=0.
$$
As was pointed out in the analysis of the two boundary problem, these optimality conditions are associated with the conditions \eqref{stationarycond1} and \eqref{stationarycond2}. The main difference is naturally the absence of the term associated with the control policy acting to the opposite direction.
It is also worth noticing that the optimality conditions \eqref{upper} and \eqref{lower} constitute the limit of the optimality condition (3.17) in Proposition 3.5 of \cite{HeStZh} as the fixed transaction/ordering costs tend zero and the costs are linear.

\subsection{Example: Controlled Ornstein-Uhlenbeck process}

In order to exemplify our general results, we now consider the downward reflection problem \eqref{sscpdown} in the case where $c(x)=|x|$, $\mu(x)=-\mu x$, and $\sigma(x)=\sigma$, where $\mu>0$ and $\sigma>0$. In that case the optimality condition reads for $b>0$ as
$$
\frac{1}{\mu}\left(2-e^{-\frac{\mu}{\sigma^2}b^2}\right)+q_d e^{-\frac{\mu}{\sigma^2}b^2} = (b-q_d \mu b)\frac{2}{\sigma}\sqrt{\frac{\pi}{\mu}}\Phi\left(\frac{\sqrt{2\mu}}{\sigma}b\right).
$$
As in the two-boundary case, we again observe that the optimal boundary is linear $b^\ast = \zeta^\ast \sigma$ as a function of volatility, where the ratio $\zeta^\ast$ constitutes the root of equation
$$
\frac{1}{\mu}\left(2-e^{-\mu{\zeta^\ast}^2}\right)+q_d e^{-\mu{\zeta^\ast}^2} = (1-q_d \mu)2\zeta^\ast\sqrt{\frac{\pi}{\mu}}\Phi\left(\sqrt{2\mu}\zeta^\ast\right).
$$
The critical ratio $\zeta^\ast$ is illustrated in Figure \ref{OUone} as a function of the drift coefficient $\mu$ under the assumption that $q_d=0.1$.
\begin{figure}[!ht]
\begin{center}
\includegraphics[width=0.5\textwidth]{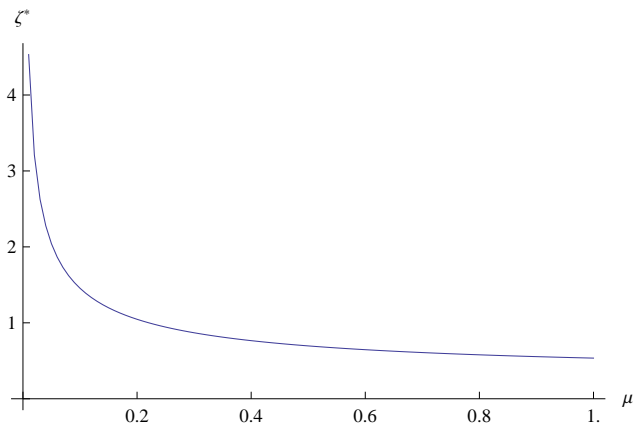}
\end{center}
\caption{\small The Critical Ratio $\zeta^\ast$}\label{OUone}
\end{figure}
As is clear from Figure \ref{OUone}, the sensitivity of the optimal boundary with respect to changes in volatility depends on the rate at which the underlying is expected to grow. The smaller the growth rate is, the higher the multiplier becomes.

\section{Conclusions}

We considered a class of ergodic singular stochastic control problems of a regular linear diffusion. We characterized the optimal policy and its value by relying on relatively elementary results from the classical theory of linear diffusions and ordinary optimization. Our results indicate that the convexity or symmetry of costs is not necessary for the existence of an optimal control characterized by two optimal thresholds at which the underlying diffusion should be reflected by utilizing a standard local time reflection policy. Our results also indicate that increased volatility expands the continuation region and increases the expected long run average costs whenever the value is convex;  a result which is in line with results of standard models focusing on the minimization of the expected cumulative present value of the costs.

The singular stochastic control setting considered in this paper is just one class of bounded variation controls arising in the literature applying stochastic control theory. As is known, it is associated with the most flexible type of that type of policies since the decision maker can adjust the path by infinitesimal amounts at any time. In many practical situations this is unfortunately not possible due to, for example, the presence of fixed controlling costs. Such situations result into impulse control models which require a slightly different analysis. However, the recent results by \cite{HeStZh} on inventory models indicate that analyzing that class of problems within the ergodic setting is doable at least in the single boundary setting. It would, therefore, be of interest to try to extend our findings to the two boundary impulse control setting. This is to my best knowledge a still open question left for future research.

\end{document}